\documentclass[12pt]{iopart}
\usepackage{amsthm}
\usepackage{iopams}
\usepackage{setstack}
\usepackage{graphics}
\usepackage{graphicx}
\usepackage{here}

\usepackage{color}
\usepackage[colorlinks]{hyperref}

\newcommand{\norm}[1]{\left\Vert#1\right\Vert}

\newcommand{\abs}[1]{\left\vert#1\right\vert}
\newcommand{\R}{{\rm I}\!{\rm R}} 

 \newtheorem{thm}{Theorem}[section]

 \newtheorem{prop}{Proposition}[section]
 \newtheorem{defn}{Definition}[section]
 \newtheorem{rem}{Remark}[section]

\begin{document}

\title[Generalized Qualification and Qualification Levels for SRM]{Generalized Qualification and Qualification Levels for Spectral Regularization Methods}

\author{T. Herdman, R. D. Spies and K. G. Temperini}

\footnote[0]{This work was supported by DARPA/SPO, NASA LaRC and the
National Institute of Aerospace under grant VT-03-1, 2535, and in
part by AFOSR Grants F49620-03-1-0243 and FA9550-07-1-0273, by
Consejo Nacional de Investigaciones Cient{\'i}ficas y T\'{e}cnicas,
CONICET, and by Universidad Nacional del Litoral, U.N.L., Argentina,
through project CAI+D 2006, P.E. 236.\\

T. Herdman, Interdisciplinary Center for Applied Mathematics,
ICAM, Virginia Tech, Blacksburg, VA 24061, USA.
E-mail: herdman@icam.vt.edu
\\

R. D. Spies, Instituto de Matem\'{a}tica Aplicada del Litoral, IMAL,
CONICET-UNL, G\"{u}emes 3450, S3000GLN, Santa Fe, Argentina.
Departamento de Matem\'{a}tica, Facultad de Ingenier\'{\i}a Qu\'{\i}mica, UNL,
Santa Fe, Argentina. E-mail: rspies@imalpde.santafe-conicet.gov.ar
\\

K. G. Temperini, IMAL, CONICET-UNL, G\"{u}emes 3450, S3000GLN, Santa
Fe, Argentina. Departamento de Matem\'{a}tica, Facultad de Humanidades
y Ciencias, UNL, Santa Fe, Argentina. \break E-mail:
ktemperini@sanatfe-conicet.gov.ar}

\begin{abstract}
The concept of qualification for spectral regularization methods
(SRM) for inverse ill-posed problems is strongly associated to the
optimal order of convergence of the regularization error
(\cite{bookEHN}, \cite{Mathe2004}, \cite{ref:Mathe-Pereverzev-2003},
\cite{Vainikko-1982}). In this article, the definition of
qualification is extended and three different levels are introduced:
weak, strong and optimal.
 It is shown that the weak qualification extends the
definition introduced by Math\'{e} and Pereverzev
(\cite{ref:Mathe-Pereverzev-2003}), mainly in the sense that the
functions associated to orders of convergence and source sets need
not be the same.  It is shown that certain methods possessing
infinite classical qualification, e.g. truncated singular value
decomposition (TSVD), Landweber's method and Showalter's method,
also have generalized qualification leading to an optimal order of
convergence of the regularization error.
Sufficient conditions for a SRM to have weak qualification are
provided and necessary and sufficient conditions for a given order
of convergence to be strong or optimal qualification are found.
Examples of all three qualification levels are provided and the
relationships between them as well as with the classical concept
of qualification  and the qualification introduced in
\cite{ref:Mathe-Pereverzev-2003} are shown. In particular, SRMs
having extended qualification in each one of the three levels and
having zero or infinite classical qualification are presented.
Finally several implications of this theory in the context of
orders of convergence, converse results and maximal source sets
for inverse ill-posed problems, are shown.
\end{abstract}
\noindent{\textbf{Keywords.}} Qualification, Regularization
method, Inverse ill-posed problem. \maketitle
\section{Introduction and preliminaries}

Let $X, Y$ be infinite dimensional Hilbert spaces and
$T:X\rightarrow Y$ a bounded linear operator. If $\mathcal{R}(T)$,
the range of $T$, is not closed it is well known that the linear
operator equation
\begin{equation}\label{eq:0}
    Tx=y
\end{equation}
is ill-posed, in the sense that $T^\dag$, the Moore-Penrose
generalized inverse of $T$, is not bounded \cite{bookEHN}. The
Moore-Penrose generalized inverse is strongly related to the
least-squares (LS) solutions of (\ref{eq:0}). In fact equation
(\ref{eq:0}) has a LS solution if and only if $y$ belongs to
$\mathcal{D}(T^\dag)$, the domain of $T^\dag$, which is defined as
$\mathcal{D}(T^\dag)\doteq \mathcal{R}(T)\oplus
\mathcal{R}(T)^\perp$. In that case, $x^\dag\doteq T^\dag y$ is
the best approximate solution (i.e. the LS solution of minimum
norm) and the set of all LS solutions of (\ref{eq:0}) is given by
$x^\dag + \mathcal{N}(T)$. If the problem is ill-posed, then
$x^\dag$ does not depend continuously on the data $y$. Hence if
instead of the exact data $y$, only an approximation $y^\delta$ is
available, with $\norm{y-y^\delta}\leq \delta$, where  $\delta>0$
is the noise level or observation error, then it is possible that
$T^\dag y^\delta$ does not exist or, if it exists, then it will
not necessarily be a good approximation of $x^\dag$, even if
$\delta$ is very small. This instability becomes evident when
trying to approximate $x^\dag$ by standard numerical methods and
procedures. Thus, for instance, except under rather restrictive
conditions (\cite{Luecke-Hickey-1985}, \cite{Vainikko-1985}), the
application of the standard LS approximations procedure on a
sequence $\{X_n\}$ of finite dimensional subspaces of $X$, whose
union is dense in $X$, will result in a sequence $\{x_n\}$ of LS
approximating solutions which does not converge to $x^\dag$ (see
\cite{ref:Seidman-80}). Moreover, this divergence can occur with
arbitrarily large speed (see \cite{Spies-Temperini-2006}).

Ill-posed problems must be regularized before pretending to
successfully attack the problem of numerically approximating their
solutions. Regularizing an ill-posed problem such as (\ref{eq:0})
essentially means approximating the operator $T^\dag$ by a
parametric family of continuous operators $\{R_\alpha\}$, where
$\alpha$ is called the regularization parameter. More precisely,
for $\alpha \in (0,\alpha_0)$ with $\alpha_0 \in (0,+\infty]$, let
$R_\alpha:Y \to X$ be a continuous (not necessarily linear)
operator. The set $\{R_\alpha\}_{\alpha \in (0,\alpha_0)}$ is said
to be a ``family of regularization operators'' (FRO) for $T^\dag$,
if for every $y \in \mathcal{D}(T^\dag)$, there exists a parameter
choice rule $\alpha=\alpha(\delta,y^\delta)$ such that
\begin{equation*}
    \underset{\delta \to 0^+}{\lim} \underset{\underset{\norm{y^\delta-y}\leq
    \delta}{y^\delta \in Y}}{\sup} \norm{R_{\alpha(\delta,y^\delta)}y^\delta-T^\dag
    y}=0.
\end{equation*}
Here the parameter choice rule $\alpha:\R^+ \times Y \to
(0,\alpha_0)$ is such that
\begin{equation*}
 \underset{\delta \to 0^+}{\lim} \underset{\underset{\norm{y^\delta-y}\leq
    \delta}{y^\delta \in Y}}{\sup} \alpha(\delta,y^\delta)=0.
\end{equation*}

If $y \in \mathcal{D}(T^\dag)$, then $x^\dag$ satisfies the normal
equation $(T^\ast T)x^\dag=T^\ast y$ and $x^\dag$ can be written
as
\begin{equation}\label{eq:int-inv}
x^\dag\doteq T^\dag y=\int_0^{\norm{T}^2+} \frac{1}{\lambda} \,
dE_\lambda T^\ast y,
\end{equation}
where $\{E_\lambda\}_{\lambda \in \R}$ is the spectral family
associated to the self-adjoint operator $T^\ast T$ (see
\cite{book:Dautray-Lions}, \cite{bookEHN}). However, since we are
assuming that $\mathcal{R}(T)$ is not closed (and therefore
$\mathcal{D}(T^\dag)\subsetneq Y$), if $y \notin
\mathcal{D}(T^\dag)$ then the integral in (\ref{eq:int-inv}) does
not exist since in that case  $0 \in \sigma(T^\ast T)$ and $\frac 1
\lambda$ has a pole at 0. Moreover in this case, the operator
$T^\dag$ defined in (\ref{eq:int-inv}) for $y \in
\mathcal{D}(T^\dag)$, is not bounded. For that reason, many
regularization methods are based on spectral theory and consist on
defining $R_\alpha\doteq \int_0^{\norm{T}^2+} g_\alpha(\lambda) \,
dE_\lambda\, T^\ast$ where $\{g_\alpha\}$ is a family of functions
appropriately defined such that for every $\lambda \in
(0,\norm{T}^2]$ there holds $\underset{\alpha \rightarrow
0^+}{\lim}g_\alpha(\lambda)=\frac{1}{\lambda}$.

Let $\{g_\alpha\}_{\alpha \in (0, \alpha_0)}$ be a parametric
family of functions $g_\alpha:[0,+\infty)\rightarrow \R$ defined
for all $\alpha \in (0,\alpha_0)$. We shall say that
$\{g_\alpha\}_{\alpha \in (0, \alpha_0)}$ is a ``spectral
regularization method'' (SRM), if it satisfies the following
hypotheses:

\textit{H1}.\; For every fixed $\alpha \in (0,\alpha_0),$
$g_\alpha(\lambda)$ is piecewise continuous with respect to
$\lambda$, for $\lambda \in [0,+\infty)$;

\textit{H2}.\; There exists a constant $C>0$ (independent of
$\alpha$) such that $\abs{\lambda g_\alpha (\lambda)}\leq C$ for
every $\lambda \in [0,+\infty)$;

\textit{H3}. \;For every $\lambda \in (0,+\infty)$,
$\underset{\alpha \rightarrow 0^+}{\lim}
g_\alpha(\lambda)=\frac{1}{\lambda}.$

It can be shown that if $\{g_\alpha\}_{\alpha \in (0, \alpha_0)}$
is a SRM then the family of operators $\{R_\alpha\}_{\alpha \in
(0, \alpha_0)}$ defined by
\begin{equation*} 
    R_\alpha\doteq \int g_\alpha(\lambda)\, dE_\lambda \,T^\ast
    =g_\alpha(T^\ast T)T^\ast,
\end{equation*}
is a FRO for $T^\dag$ (\cite{bookEHN}, Theorem 4.1). In this case
we shall say that $\{R_\alpha\}_{\alpha \in (0, \alpha_0)}$ is a
``spectral regularization family" for $T^\dag$. The use of this
terminology has to do with the fact that each one of its elements
is defined in terms of an integral with respect to the spectral
family $\{E_\lambda\}_{\lambda\in \R}$ associated to the operator
$T^\ast T$. Note that given the operator $T$, it is sufficient
that $g_\alpha(\lambda)$ be defined for $\lambda \in
[0,\norm{T}^2]$, since $E_\lambda$ is ``constant" outside that
interval.

It is well known that for ill-posed problems it is not possible to
reconstruct the exact solution $x^\dag$ with any degree of
accuracy unless additional \textit{a-priori} information about
$x^\dag$ is available (\cite{Spies-Temperini-2006}, \cite{bookEHN}
Proposition 3.11). On the other hand, given certain
\textit{a-priori} information about $x^\dag$, it could be
desirable to know the best order of convergence (of the
regularization error $\norm{R_\alpha y - x^\dag}$ as a function of
the regularization parameter $\alpha$, or of the total error
$\norm{R_\alpha y^\delta - x^\dag}$ as a function of the noise
level $\delta$), that can be achieved with a regularization method
under those \textit{a-priori} assumptions. Conversely, given an
order of convergence, one could be interested in determining the
possible existence of ``source sets" on which a certain
regularization method reaches that order of convergence. In this
case it could further be of interest to determine ``maximal source
sets". All these problems are strongly related to the concepts of
qualification and saturation of a regularization method
(\cite{bookEHN}, \cite{Engl-Kunisch-Neubauer-1989},
\cite{Mathe2004}, \cite{ref:Mathe-Pereverzev-2003},
\cite{Neubauer94}, \cite{ref:Neubauer-97}).

In \cite{Vainikko-1982} the notion of qualification of a
regularization method was introduced for the first time and the
decisive role of this concept in relation to the order of
convergence of the regularization error was shown. In the sequel,
we shall simply denote with $\{g_\alpha\}$ the SRM
$\{g_\alpha\}_{\alpha \in (0,\alpha_0)}$. We now recall the
definition of classical qualification for SRMs (see
\cite{bookEHN}).

\begin{defn}\label{def:calif-clasica}
Let $\{g_\alpha\}$ be a SRM and denote with
$\mathcal{I}(g_\alpha)$ the set
$$\mathcal{I}(g_\alpha)\doteq\{\mu\geq 0:\, \forall\, \lambda \in
[0,+\infty),\exists\, k>0 \textrm{ such that
}\lambda^\mu\abs{1-\lambda g_\alpha(\lambda)}\leq k\,\alpha^\mu\,
, \forall\, \alpha \in (0,\alpha_0)\}$$ and let $\mu_0 \doteq
\underset{\mu \in \mathcal{I}(g_\alpha)}{\sup}\, \mu$. If $
0<\mu_0<+\infty$, we say that $\{g_\alpha\}$ has classical
qualification and in that case the number $\mu_0$ is called
``order" of the classical qualification.
\end{defn}

\begin{rem}
Note that $0 \in \mathcal{I}(g_\alpha)$ by virtue of \textit{H2}
and therefore $\mathcal{I}(g_\alpha)$ is always nonempty.
\end{rem}

In \cite{ref:Mathe-Pereverzev-2003} Math\'{e} and Pereverzev first
introduced the following definition of qualification for a
spectral regularization method, formalizing and extending the
classical notion of the concept.

\begin{defn}\label{def:calif-mathe}
Let $\rho:(0,a]\rightarrow (0,\infty)$ be an increasing function.
It is said that the regularization method $\{g_\alpha\}$ has
qualification $\rho$ if there exists a constant $\gamma \in
(0,\infty)$ such that
\begin{equation}\label{eq:calif-mathe}
    \underset{\lambda\in (0,a]}{\sup}\abs{1-\lambda
    g_\alpha(\lambda)}\rho(\lambda)\leq \gamma \,\rho(\alpha)\quad
    \forall \;\alpha \in (0,a].
\end{equation}
\end{defn}

\bigskip

In this article we generalize the previous concept, mainly by
allowing the function $\rho(\lambda)$ appearing in the left hand
side of (\ref{eq:calif-mathe}) to be substituted by a general
function $s(\lambda)$ with similar properties.

\begin{rem}
It is important to point out that in \cite{bookEHN} the
``classical qualification" of a method was defined to be the
number $\mu_0$ in Definition \ref{def:calif-clasica} (even in the
case $\mu_0=\infty$). However, from our point of view the
``generalized qualification"  of a method will not be a number but
rather a function of the regularization parameter $\alpha$ as an
order of convergence in the sense of Definition
\ref{def:calif-mathe}. In the case of SRMs with classical
qualification of positive finite order $\mu_0$, the corresponding
generalized qualification will be shown to be the function
$\rho(\alpha)=\alpha^{\mu_0}$, coinciding with the classical
approach. Since in the extreme cases $\mu_0=0$ and $\mu_0=\infty$
that function does not define an order of convergence, we have
preferred to exclude them from the definition of classical
qualification (Definition \ref{def:calif-clasica}) and,
accordingly, we shall say that the method does not have classical
qualification.

\end{rem}

The organization of this article is as follows. In Section 2 the
concepts of weak and strong source-order pair and of order-source
pair are defined and three qualification levels for SRM are
introduced: weak, strong and optimal. A sufficient condition for
the existence of weak qualification is provided and necessary and
sufficient conditions for an order of convergence to be strong or
optimal qualification are given. In Section 3, examples of all
qualification levels are provided and the relationships between
them and with the classical qualification and the qualification
introduced in \cite{ref:Mathe-Pereverzev-2003} are shown. In
particular, SRMs having qualification in each one of the three
levels and not having classical qualification are presented.
Finally  several implications of this theory in the context of
orders of convergence, converse results and maximal source sets
for inverse ill-posed problems are shown in Section 4.


\section{Source-order and order-source pairs. Generalized qualification and qualification levels.}

 It is well known that there exist SRMs for which the corresponding
$\mu_0$ given in Definition \ref{def:calif-clasica} is infinity,
e.g.{\;\,}truncated singular value decomposition (TSVD),
Landweber's method and Showalter's method. However, a careful
analysis leads to observe that the concept of qualification as
optimal order of convergence of the regularization error remains
alive underlying most of these and many other methods. In this
section we generalize the definition of qualification introduced
by Math\'{e}-Pereverzev in \cite{ref:Mathe-Pereverzev-2003} and
thereby the notion of classical qualification of a SRM. Also three
different levels of qualification are introduced: weak, strong and
optimal. These levels introduce natural hierarchical categories
for the SRMs and we show that the generalized qualification
corresponds to the lowest of these levels. Moreover, a sufficient
condition which guarantees that a SRM possesses qualification in
the sense of this generalization is provided and necessary and
sufficient conditions for a given order of convergence to be
strong or optimal qualification are found.

We denote with $\mathcal{O}$ the set of all non decreasing
functions $\rho:\R^+ \to \R^+$ such that $\underset{\alpha
\rightarrow 0^+}{\lim}\rho(\alpha)=0$ and with $\mathcal{S}$ the
set of all continuous functions $s:\R^+_0 \to \R^+_0$ satisfying
$s(0)=0$ and such that $s(\lambda)>0$ for every $\lambda>0.$ If
moreover $s$ is increasing, then it is an \textit{index function}
in the sense of Math\'{e}-Pereverzev
(\cite{ref:Mathe-Pereverzev-2003}).

\begin{defn}
Let $\rho,\tilde{\rho}\in \mathcal{O}$. We say that ``$\rho$
precedes $\tilde{\rho}$ at the origin'' and we denote it with
$\rho \preceq \tilde{\rho}$, if there exist positive constants $c$
and $\varepsilon$ such that $\rho(\alpha)\leq
c\,\tilde{\rho}(\alpha)$ for every $\alpha \in (0,\varepsilon)$.
\end{defn}

\begin{defn} \label{def:order}
Let $\rho,\tilde{\rho}\in \mathcal{O}$. We say that ``$\rho$ and
$\tilde{\rho}$ are equivalent at the origin'' and we denote it
with $\rho \approx\tilde{\rho}$, if they precede each other at the
origin, that is, if there exist constants $\varepsilon, c_1, c_2$,
$\varepsilon>0$, $0<c_1<c_2<\infty$ such that $c_1\,
\rho(\alpha)\leq \tilde{\rho}(\alpha) \leq c_2\, \rho(\alpha)$ for
every $\alpha \in (0,\varepsilon)$.
\end{defn}

Clearly, ``$\approx$" introduces an order of equivalence in
$\mathcal{O}$. Analogous definitions and notation will be used for
$s,\tilde{s} \in \mathcal{S}$.

\begin{defn}
Let $\{g_\alpha\}$ be a SRM, $r_{\alpha}(\lambda)\doteq 1-\lambda
g_{\alpha}(\lambda)$, $\rho\in \mathcal{O}$ and $s \in
\mathcal{S}.$

\textbf{i)} We say that $(s,\rho)$ is a ``weak source-order pair for
$\{g_\alpha\}$'' if it satisfies
\begin{equation}\label{eq:O}
\frac{s(\lambda)\abs{r_\alpha(\lambda)}}{\rho(\alpha)}=O(1)\quad
\textrm{for}\;\; \alpha \rightarrow 0^+, \; \forall\; \lambda>0.
\end{equation}

\textbf{ii)} We say that $(s, \rho)$ is a ``strong source-order
pair for $\{g_\alpha\}$'' if it is a weak source-order pair and
there is no $\lambda>0$ for which $O(1)$ in {\rm(\ref{eq:O})} can
be replaced by $o(1)$. That is, if {\rm(\ref{eq:O})} holds and
also
\begin{equation}\label{eq:no-o}
\underset{\alpha \rightarrow
0^+}{\limsup}\,\frac{s(\lambda)\abs{r_\alpha(\lambda)}}{\rho(\alpha)}>0\quad
\forall\; \lambda>0.
\end{equation}

\textbf{iii)} We say that $(\rho, s)$ is an ``order-source pair
for $\{g_\alpha\}$'' if there exist a constant $\gamma>0$ and a
function $h:(0,\alpha_0)\rightarrow \R^+$ with $\underset{\alpha
\rightarrow0^+}{\lim}h(\alpha)=0$, such that
\begin{equation}\label{eq:4.50}
\frac{s(\lambda)\abs{r_\alpha(\lambda)}}{\rho(\alpha)}\geq \gamma
\quad \forall\; \lambda \in [h(\alpha),+\infty).
\end{equation}
\end{defn}

In the previous definitions we shall refer to the function $\rho$
as the ``order of convergence" and to $s$ as the ``source
function". The reason for using this terminology will become clear
in Section 4 when we shall see applications of these concepts in
the context of direct and converse results for regularization
methods.

The following observations follow immediately from the
definitions.
\begin{enumerate}

\item If $(s,\rho)$ is a weak source-order pair for
$\{g_\alpha\}$ which is not a strong source-order pair, then there
exists $\lambda_0>0$ such that $\underset{\alpha \rightarrow
0^+}{\limsup}\,\frac{s(\lambda_0)\abs{r_\alpha(\lambda_0)}}{\rho(\alpha)}=0$
and therefore $(\rho,s)$ cannot be an order-source pair for
$\{g_\alpha\}$. Thus  if $(\rho, s)$ is an order-source pair and
$(s,\rho)$ is a weak source-order pair, then $(s,\rho)$ is further
a strong source-order pair in the sense of \textbf{\textit{ii})}.

\item Let $\rho,\tilde{\rho}\in \mathcal{O}$.
\begin{enumerate}
\item  If $(s,\rho)$ is a weak source-order pair for
$\{g_\alpha\}$ and $\rho \preceq \tilde{\rho}$ then
$(s,\tilde{\rho})$ is also a weak source-order pair for
$\{g_\alpha\}$.
\item If  $(s,\rho)$ is a weak source-order pair for $\{g_\alpha\}$
and $\tilde{s}\in \mathcal{S}$ is such that there exists $c>0$ for
which $\tilde{s}(\lambda)\leq c\,s(\lambda)$ for every
$\lambda>0$, then $(\tilde{s},\rho)$ is also a weak source-order
pair for $\{g_\alpha\}$.
\end{enumerate}
\end{enumerate}

In the following definition we introduce the concept of
generalized qualification and three different levels of it.

\begin{defn}\label{def:calif-3} Let $\{g_\alpha\}$ be a SRM.

\textbf{i)} We say that $\rho$ is ``weak or generalized
qualification of $\{g_\alpha\}$'' if there exists a function $s$
such that $(s,\rho)$ is a weak source-order pair for
$\{g_\alpha\}$.

\textbf{ii)} We say that $\rho$ is ``strong qualification of
$\{g_\alpha\}$'' if there exists a function $s$ such that
$(s,\rho)$ is a strong source-order pair for $\{g_\alpha\}$.

\textbf{iii)} We say that $\rho$ is ``optimal qualification of
$\{g_\alpha\}$'' if there exists a function $s$ such that
$(s,\rho)$ is a strong source-order pair for $\{g_\alpha\}$ (it is
sufficient that $(s,\rho)$ be a weak source-order pair) and
$(\rho, s)$ is an order-source pair for $\{g_\alpha\}$.
\end{defn}

It is important to observe that weak qualification generalizes the
concept of qualification introduced by Math\'{e} and Pereverzev in
\cite{ref:Mathe-Pereverzev-2003} and therefore, the notion of
classical qualification. In fact, if $\{g_\alpha\}$ has continuous
qualification $\rho(\alpha)$ in the sense of Definition
\ref{def:calif-mathe} and $\underset{\alpha \rightarrow
0^+}{\lim}\rho(\alpha)=0$, then the function
\begin{equation}
\tilde{\rho}(\alpha)\doteq \left\{%
\begin{array}{ll}
    0, & \hbox{si $\alpha=0$;} \\
    \rho(\alpha), & \hbox{si $0<\alpha\leq a$;} \\
    \rho(a), & \hbox{si $\alpha > a$.} \\
\end{array}%
\right.
\end{equation}
is weak qualification of $\{g_\alpha\}$. However, these two
notions are not equivalent. We shall see later on that it is
possible for a function to be weak qualification of a SRM and not
be qualification according to Definition \ref{def:calif-mathe}
(see comments at the end of Section 3).

It is timely to note here that if $\{g_\alpha\}$ has classical
qualification of order $\mu_0$, then $\rho(\alpha)=\alpha^\mu$ is
weak qualification of $\{g_\alpha\}$ and moreover
$(\lambda^{\mu},\alpha^{\mu})$ is a weak source-order pair for
$\{g_\alpha\}$ for every $\mu\in (0,\mu_0]$. Conversely, if for
$\mu>0$, $(\lambda^{\mu},\alpha^{\mu})$ is a weak source-order
pair for $\{g_\alpha\}$, then this method has classical
qualification (of order $\mu_0 \geq \mu$) provided that
$\mu_0\doteq \sup$ {\Large\{}$\mu: (\lambda^{\mu},\alpha^{\mu})$
is a weak source-order pair for
$\{g_\alpha\}${\Large\}}$<+\infty$.

The following result provides a sufficient condition for the
existence of weak qualification of a SRM.

\begin{thm} \label{teo:cond-caldebil}
Let $\{g_\alpha\}$ be a SRM such that for every fixed $\lambda>0$,
$g_\alpha(\lambda)$ is decreasing in $\alpha$, for $\alpha \in
(0,\alpha_0)$.

\textbf{a)} If there exist an increasing function
$h:(0,\alpha_0)\rightarrow \R^+$ with $\underset{\alpha
\rightarrow0^+}{\lim}h(\alpha)=0$, $\rho^\ast \in \mathcal{O}$ and
$\varepsilon>0$ such that for every $\alpha \in (0,\varepsilon),$
\begin{equation}\label{eq:cond-califdebil}
\underset{\lambda \in
[h(\alpha),+\infty)}{\sup}\abs{r_\alpha(\lambda)}\leq \rho^\ast
(\alpha),
\end{equation}
then $\{g_\alpha\}$ has weak qualification and in that case
$\rho^\ast$ is weak qualification of the method.

\textbf{b)} If for every $\alpha \in (0,\alpha_0)$,
$r_\alpha(\lambda)$ is positive and monotone decreasing for
$\lambda \in (0,+\infty)$, then it is always possible to find $h$
and $\rho^\ast$ as in \textbf{a)} satisfying \text{\rm
(\ref{eq:cond-califdebil})} for all $\alpha \in (0,\alpha_0)$.

\end{thm}
\begin{proof}
\textbf{a)} Let $h:(0,\alpha_0)\rightarrow \R^+$ be an increasing
function with $\underset{\alpha \rightarrow0^+}{\lim}h(\alpha)=0$,
$\rho^\ast \in \mathcal{O}$ and $\varepsilon>0$ such that for
every $\alpha \in (0,\varepsilon)$ condition
(\ref{eq:cond-califdebil}) holds.

\underline{Case I}: there exists $\tilde{\alpha}\in
(0,\varepsilon)$ such that $\underset{\lambda \in \,
[h(\tilde{\alpha}),\,+\infty)}{\sup}\abs{r_{\tilde{\alpha}}(\lambda)}>0$.

Since $h(\alpha)$ is increasing, it follows that
$\underset{\lambda \in \,
[h(\alpha),\,+\infty)}{\sup}\abs{r_{\alpha}(\lambda)}>0$ for every
$\alpha \in (0,\tilde{\alpha}]$. Let $\lambda_0>0$. Then for every
$\alpha \in (0,\tilde{\alpha}]$,
\begin{equation}\label{eq:12}
    \frac{\abs{r_\alpha(\lambda_0)}}{\rho^\ast(\alpha)}\leq  \frac{\abs{r_\alpha(\lambda_0)}}{\underset{\lambda \in
[h(\alpha),+\infty)}{\sup}\abs{r_\alpha(\lambda)}}.
\end{equation}
Since $\underset{\alpha \rightarrow0^+}{\lim}h(\alpha)=0$, there
exists $\alpha^\ast \in (0,\tilde{\alpha})$ such that $\lambda_0
\in [h(\alpha),+\infty)$ for every $\alpha \in (0,\alpha^\ast],$
from which it follows that for every $\alpha \in (0,\alpha^\ast]$,

\begin{equation}\label{eq:11}
\frac{\abs{r_\alpha(\lambda_0)}}{\underset{\lambda \in
[h(\alpha),+\infty)}{\sup}\abs{r_\alpha(\lambda)}}\leq 1.
\end{equation}
>From (\ref{eq:12}) and (\ref{eq:11}) it follows that for every
$\lambda_0>0$
$$\underset{\alpha \to
0^+}{\limsup}\frac{\abs{r_\alpha(\lambda_0)}}{\rho^\ast(\alpha)}\leq
1.$$ Then, for any bounded $s \in \mathcal{S}$  the pair
$(s,\rho^\ast)$ satisfies (\ref{eq:O}), i.e., it is a weak
source-order pair for $\{g_\alpha\}$. Thus  we have proved that
$\rho^\ast$ is weak qualification of $\{g_\alpha\}$.

\underline{Case II}: $\underset{\lambda \in
[h(\alpha),+\infty)}{\sup}\abs{r_\alpha(\lambda)}=0$ for every
$\alpha \in (0,\varepsilon)$.

Let $\lambda_0>0$. Since $\underset{\alpha
\rightarrow0^+}{\lim}h(\alpha)=0$, there exists $\alpha^\ast \in
(0,\varepsilon)$ such that $\lambda_0 \in [h(\alpha),+\infty)$ for
every $\alpha \in (0,\alpha^\ast].$ Then
$\abs{r_\alpha(\lambda_0)}\leq \underset{\lambda \in
[h(\alpha),+\infty)}{\sup}\abs{r_\alpha(\lambda)}=0$ for every
$\alpha \in (0,\alpha^\ast)$, from what it follows that
$r_\alpha(\lambda_0)=0$. Then, for any $s \in \mathcal{S}$,
\begin{equation*}
    \frac{s(\lambda_0)r_\alpha(\lambda_0)}{\rho^\ast(\alpha)}=0
    \quad \textrm{for all} \quad \alpha \in (0,\alpha^\ast).
\end{equation*}
Therefore, $(s,\rho^\ast)$ is a weak source-order pair for
$\{g_\alpha\}$, which implies that $\rho^\ast$ is weak
qualification of $\{g_\alpha\}$. (Note that in this case any
$\rho^\ast \in \mathcal{O}$ is weak qualification of
$\{g_\alpha\}$.)

\textbf{b)} Let $\{g_\alpha\}$ be a SRM such that for every
$\alpha \in (0,\alpha_0)$, $r_\alpha(\lambda)$ is positive and
monotone decreasing for $\lambda \in (0,+\infty)$. For $\lambda>0$
we define $ f(\lambda)\doteq (1-e^{-\lambda})\theta(\lambda)$,
where
\begin{equation*}
    \theta(\lambda)\doteq \sup \{\gamma \in (0,\alpha_0):r_\alpha(\lambda)\leq \lambda \;\forall \; \alpha \in
    (0,\gamma)\}.
\end{equation*}
Since for every $\lambda>0$, $\underset{\alpha\to
0^+}{\lim}r_\alpha(\lambda)=0$, it follows that given $\lambda>0$
there exists $\gamma=\gamma(\lambda)>0$ such that
$r_\alpha(\lambda)\leq \lambda$ for every $\alpha \in (0,\gamma)$.
Then $\theta(\lambda)\neq -\infty$, moreover $\theta(\lambda)\in
(0,\alpha_0]$ for every $\lambda>0$ and therefore, $f(\lambda) \in
(0,\alpha_0)$ for every $\lambda>0$. On the other hand, since for
every $\alpha \in (0,\alpha_0)$, $r_\alpha(\lambda)$ is decreasing
for $\lambda>0$, it follows immediately that $f$ is strictly
increasing. Furthermore, since $f$ is bounded, it has countably
many jump discontinuity points. Therefore, it is possible to
assume, without loss of generality, that $f$ is continuous (since,
if it is not, we can redefine it in such a way that it be
continuous, by subtracting the jumps at the discontinuity points).

Thus  $f:\R^+ \to (0,\alpha_0)$ is continuous, strictly increasing
with $\underset{\lambda \rightarrow 0^+}{\lim}f(\lambda)=0$.
Therefore, its inverse function $f^{-1}$ exists over the range of
$f$ and it is strictly increasing and continuous with
$\underset{\alpha \rightarrow0^+}{\lim}f^{-1}(\alpha)=0$. It is
possible to extend $f^{-1}$ to $(0,\alpha_0)$ in such a way that
it preserves all these properties. We shall denote with $h$ this
extension.

For $\alpha \in (0,\alpha_0)$, we define $z(\alpha)\doteq
\underset{\lambda \in
[h(\alpha),+\infty)}{\sup}\abs{r_\alpha(\lambda)}=r_\alpha(h(\alpha)).$
Since for every $\alpha \in (0,\alpha_0)$, $r_\alpha(\lambda)$ is
positive for all $\lambda>0$, it follows that $z(\alpha)$ is also
positive. Since for every $\lambda>0$, $f(\lambda)<
\theta(\lambda)$, the definition of $\theta(\lambda)$ implies that
$r_{f(\lambda)}(\lambda)\leq \lambda$ for every $\lambda>0$, or
equivalently, $r_\alpha(h(\alpha))\leq h(\alpha)$ for every
$\alpha \in (0,\alpha_0)$. Then $0<z(\alpha)\leq h(\alpha)$ for
every $\alpha \in (0,\alpha_0)$ and the fact that
$\underset{\alpha \rightarrow0^+}{\lim}h(\alpha)=0$ implies that
$\underset{\alpha \rightarrow0^+}{\lim}z(\alpha)=0$. If further
$z$ is a non decreasing function, then $z \in \mathcal{O}$ and it
suffices to define $\rho^\ast\doteq z$. On the contrary, since $z$
is bounded and positive with $\underset{\alpha
\rightarrow0^+}{\lim}z(\alpha)=0$, there always exists a function
$\rho^\ast \in \mathcal{O}$ such that $z(\alpha)\leq
\rho^\ast(\alpha)$ for every  $\alpha \in (0,\alpha_0)$, as we
wanted to show.
\end{proof}

>From the previous Theorem, it follows that the SRMs
 $\{g_\alpha\}$ such that for every $\lambda>0$,
$g_\alpha(\lambda)$ is decreasing for $\alpha \in (0,\alpha_0)$
and for every $\alpha \in (0,\alpha_0)$, $r_\alpha(\lambda)$ is
positive and decreasing for $\lambda>0$, do possess weak
qualification. It is important to observe that most of the usual
SRMs do in fact satisfy these conditions. In particular this is so
for Landweber's and Showalter's methods.

Now  given the SRM $\{g_\alpha\}$ and $\rho\in \mathcal{O}$, we
define
\begin{equation}\label{s-rho}
s_\rho(\lambda)\doteq  \underset{\alpha \to 0^+}{\liminf}
\frac{\rho(\alpha)}{\abs{r_\alpha(\lambda)}}\quad \textrm{for}
\quad \lambda \geq 0.
\end{equation}
Note that $s_\rho(0)=0.$

In the next three results we will see that the characteristics of
a given function $\rho \in \mathcal{O}$,  as a possible strong or
optimal qualification of a SRM, can be determined from properties
of that function $s_\rho$.

\begin{prop}\label{teo:cond-calif}(Necessary and sufficient
condition for strong qualification.) A function $\rho \in
\mathcal{O}$ such that $s_\rho \in \mathcal{S}$ is strong
qualification of $\{g_\alpha\}$ if and only if
\begin{equation}\label{eq:cond-calif}
0<s_\rho(\lambda)<+\infty\quad  \textrm{for every}\quad \lambda
>0.
\end{equation}
\end{prop}
\begin{proof}
Suppose that $\rho$ is strong qualification of $\{g_\alpha\}$.
Then there exists a function $s \in \mathcal{S}$ such that
$(s,\rho)$ is a strong source-order pair for $\{g_\alpha\}$. Then,
for every $\lambda>0$,
\begin{equation*}
    s_\rho(\lambda)=\underset{\alpha \to 0^+}{\liminf}
\frac{\rho(\alpha)}{\abs{r_\alpha(\lambda)}}=\frac{1}{\underset{\alpha
\to
0^+}{\limsup}\frac{\abs{r_\alpha(\lambda)}}{\rho(\alpha)}}=\frac{s(\lambda)}{\underset{\alpha
\to
0^+}{\limsup}\frac{s(\lambda)\abs{r_\alpha(\lambda)}}{\rho(\alpha)}}.
\end{equation*}
Thus  (\ref{eq:cond-calif}) follows from (\ref{eq:O}) and
(\ref{eq:no-o}).

Conversely, suppose now that $0<s_\rho(\lambda)<+\infty$ for every
$\lambda>0$. We will show that $\rho$ is strong qualification of
$\{g_\alpha\}$. For that let us see that $(s_\rho, \rho)$ is a
strong source-order pair for $\{g_\alpha\}$. Since
$0<s_\rho(\lambda)<+\infty$ for every $\lambda>0$, it follows that
\begin{equation*}
\underset{\alpha \rightarrow
0^+}{\limsup}\,\frac{s_\rho(\lambda)\abs{r_\alpha(\lambda)}}{\rho(\alpha)}=s_\rho(\lambda)\,\underset{\alpha
\rightarrow
0^+}{\limsup}\,\frac{\abs{r_\alpha(\lambda)}}{\rho(\alpha)}=1
\quad \forall\; \lambda>0.
\end{equation*}
Then, $s_\rho$ verifies  (\ref{eq:O}) and (\ref{eq:no-o}), which,
together with the fact that $s_\rho \in \mathcal{S}$, implies that
$(s_\rho,\rho)$ is a strong source-order pair and thus  $\rho$ is
strong qualification of $\{g_\alpha\}$.
\end{proof}

\begin{prop}\label{teo:nu<nu-rho} Let $\rho \in \mathcal{O}$ be
strong qualification of $\{g_\alpha\}$ and $s \in \mathcal{S}$.
Then $(s,\rho)$ is a strong source-order pair for $\{g_\alpha\}$
if and only if there exists $k>0$ such that $s(\lambda)\leq k\,
s_\rho(\lambda)$ for every $\lambda>0$.
\end{prop}

\begin{proof}
Since $\rho$ is strong qualification, by Proposition
\ref{teo:cond-calif} it follows that $s_\rho(\lambda)>0$ for every
$\lambda>0$ . Suppose now that $(s,\rho)$ is a strong source-order
pair for $\{g_\alpha\}$. Then there exist positive constants $k$
and $\varepsilon$ such that
$\frac{s(\lambda)\abs{r_\alpha(\lambda)}}{\rho(\alpha)}\leq k$ for
every $\lambda>0$, $\alpha\in (0,\varepsilon)$.  Then, for every
$\lambda>0$
\begin{equation*}
\frac{s(\lambda)}{s_\rho(\lambda)}=s(\lambda)\,\underset{\alpha
\rightarrow
0^+}{\limsup}\,\frac{\abs{r_\alpha(\lambda)}}{\rho(\alpha)}=\underset{\alpha
\rightarrow
0^+}{\limsup}\,\frac{s(\lambda)\abs{r_\alpha(\lambda)}}{\rho(\alpha)}\leq
k,
\end{equation*}
and therefore $s(\lambda)\leq k\,s_\rho(\lambda)$ for every
$\lambda>0$.

Conversely, suppose that there exists $k>0$ such that
$s(\lambda)\leq k \,s_\rho(\lambda)$ for every  $\lambda
>0$. Since $s_\rho(\lambda)>0$, it then follows that
\begin{equation*}
k \geq \frac{s(\lambda)}{s_\rho(\lambda)}= \underset{\alpha
\rightarrow
0^+}{\limsup}\,\frac{s(\lambda)\abs{r_\alpha(\lambda)}}{\rho(\alpha)}
\quad \forall\; \lambda>0,
\end{equation*}
that is, $(s,\rho)$ is a weak source-order pair for
$\{g_\alpha\}$. Moreover since $s(\lambda)$ and $s_\rho(\lambda)$
are positive for all $\lambda>0$, it follows that $s(\lambda)$
verifies (\ref{eq:no-o}) and therefore $(s, \rho)$ is,
furthermore, a strong source-order pair for $\{g_\alpha\}$.
\end{proof}

\begin{thm}\label{teo:cond-calopt}
(Necessary and sufficient condition for optimal qualification.) A
function $\rho \in \mathcal{O}$ such that $s_\rho \in \mathcal{S}$
is optimal qualification of $\{g_\alpha\}$ if and only if $s_\rho$
verifies \text{\rm (\ref{eq:4.50})} and
\text{\rm(\ref{eq:cond-calif})}.
\end{thm}

\begin{proof}
Suppose that $\rho$ is optimal qualification. Then $\rho$ is
strong qualification  and it follows from Proposition
\ref{teo:cond-calif} that $s_\rho$ verifies (\ref{eq:cond-calif}).
Moreover since $\rho$ is optimal qualification, there exists $s
\in \mathcal{S}$ such that $(s, \rho)$ is a strong source-order
pair and $(\rho, s)$ is an order-source pair. From the latter it
follows that there exist a constant $\gamma>0$ and a function
$h:(0,\alpha_0)\rightarrow \R^+$ with $\underset{\alpha
\rightarrow0^+}{\lim}h(\alpha)=0$, such that
\begin{equation}\label{eq:2}
\frac{s(\lambda)\abs{r_\alpha(\lambda)}}{\rho(\alpha)}\geq \gamma
\quad \forall\; \lambda \in [h(\alpha),+\infty).
\end{equation}
On the other hand, since $(s, \rho)$ is a strong source-order pair
for $\{g_\alpha\}$, it follows from Proposition
\ref{teo:nu<nu-rho} that there exists $k>0$ such that
\begin{equation}\label{eq:3}
s(\lambda)\leq k\, s_\rho(\lambda)\quad \textrm{for every}\;
\lambda>0.
\end{equation}
>From (\ref{eq:2}) and (\ref{eq:3}) it follows that
\begin{equation*}
\frac{s_\rho(\lambda)\abs{r_\alpha(\lambda)}}{\rho(\alpha)}\geq
\frac{\gamma}{k} \quad \forall\; \lambda \in [h(\alpha),+\infty),
\end{equation*}
that is, $s_\rho$ satisfies (\ref{eq:4.50}) as we wanted to show.

Conversely, suppose that $s_\rho \in \mathcal{S}$ verifies
(\ref{eq:4.50}) and (\ref{eq:cond-calif}). By Proposition
\ref{teo:cond-calif} we have that $(s_\rho, \rho)$ is a strong
source-order pair for $\{g_\alpha\}$ and (\ref{eq:4.50}) implies
that $(\rho,s_\rho)$ is an order-source pair. Then, $\rho$ is
optimal qualification of $\{g_\alpha\}$.
\end{proof}

Next we will show the uniqueness of the source function.
\begin{thm}\label{teo:unica nu}
If $\rho$ is optimal qualification of $\{g_\alpha\}$ then there
exists at most one function $s$ (in the sense of the equivalence
classes induced by Definition \ref{def:order}) such that
$(s,\rho)$ is a strong source-order pair and $(\rho,s)$ is an
order-source pair for $\{g_\alpha\}$. Moreover if $s_\rho \in
\mathcal{S}$, then $s_\rho$ is such a unique function.
\end{thm}

\begin{proof}
Given that $\rho$ is optimal qualification of $\{g_\alpha\}$,
there exists at least one function $s$ such that $(s, \rho)$ is a
strong source-order pair and $(\rho, s)$ is an order-source pair
for $\{g_\alpha\}$. Suppose now that there exist $s_1$ and $s_2$
such that $(s_1,\rho)$ and $(s_2,\rho)$ are strong source-order
pairs and $(\rho,s_1)$ and $(\rho,s_2)$ are order-source pairs for
$\{g_\alpha\}$. Then there exist $\gamma>0$ and a function
$h:(0,\alpha_0)\to \R^+$ with $\underset{\alpha
\rightarrow0^+}{\lim}h(\alpha)=0$, such that
$\frac{s_2(\lambda)\abs{r_\alpha(\lambda)}}{\rho(\alpha)}\geq
\gamma$ for every $\lambda \in [h(\alpha),+\infty)$. Then,
\begin{equation}\label{eq:4}
\fl
s_1(\lambda)=\frac{\frac{s_1(\lambda)s_2(\lambda)\abs{r_\alpha(\lambda)}}{\rho(\alpha)}}{\frac{s_2(\lambda)\abs{r_\alpha(\lambda)}}{\rho(\alpha)}}\leq
\frac{s_2(\lambda)}{\gamma}\frac{s_1(\lambda)\abs{r_\alpha(\lambda)}}{\rho(\alpha)}\quad
\forall\, \lambda \in [h(\alpha),+\infty), \forall\, \alpha \in
(0,\alpha_0).
\end{equation}
On the other hand, since $(s_1,\rho)$ is a strong source-order
pair, there exist positive constants $k$ and $\varepsilon$ such
that
\begin{equation}\label{eq:5}
\frac{s_1(\lambda)\abs{r_\alpha(\lambda)}}{\rho(\alpha)}\leq
k\quad \forall \,\lambda>0, \forall\, \alpha\in (0,\varepsilon).
\end{equation}
From (\ref{eq:4}) and (\ref{eq:5}) it follows that
\begin{equation*}
s_1(\lambda)\leq \frac{k}{\gamma}\,s_2(\lambda)\quad \forall
\,\lambda \in [h(\alpha),+\infty), \forall\, \alpha\in
(0,\varepsilon).
\end{equation*}
Since $\underset{\alpha \rightarrow0^+}{\lim}h(\alpha)=0$ we have
that $s_1(\lambda)\leq \frac{k}{\gamma}\,s_2(\lambda)$ for every
$\lambda>0$. Analogously, by interchanging $s_1$ and $s_2$ it
follows that there exists $\tilde{k}>0$ such that
$s_2(\lambda)\leq \tilde{k}\,s_1(\lambda)$ for every $\lambda>0$
and therefore, $s_1 \approx s_2$.

Suppose now that $s_\rho \in \mathcal{S}$. Since $\rho$ is optimal
qualification of $\{g_\alpha\}$ it follows from Theorem
\ref{teo:cond-calopt} that $s_\rho$ verifies (\ref{eq:4.50}) and
(\ref{eq:cond-calif}). Then, $s_\rho$ is the unique function such
that $(s_\rho,\rho)$ is a strong source-order pair and
$(\rho,s_\rho)$ is an order-source pair for $\{g_\alpha\}$.
\end{proof}

The following is a result about the uniqueness of the order.
\begin{thm}
If $(s,\rho_1)$ and $(s,\rho_2)$ are strong source-order pairs for
$\{g_\alpha\}$ and there exists $\underset{\alpha \to
0^+}{\lim}\frac{\rho_1(\alpha)}{\rho_2(\alpha)}$, then
$\rho_1\approx \rho_2$.
\end{thm}

\begin{proof}
Suppose that $(s,\rho_1)$ and $(s,\rho_2)$ are strong source-order
pairs for $\{g_\alpha\}$. We will first show that
$\underset{\alpha \to
0^+}{\limsup}\frac{\rho_1(\alpha)}{\rho_2(\alpha)}>0.$ Suppose
that
\begin{equation}\label{eq:6}
\underset{\alpha \to
0^+}{\limsup}\frac{\rho_1(\alpha)}{\rho_2(\alpha)}=0.
\end{equation}
Since $(s,\rho_1)$ is a strong source-order pair we have that
\begin{equation}\label{eq:7}
\frac{s(\lambda)\abs{r_\alpha(\lambda)}}{\rho_1(\alpha)}=O(1)\quad
\textrm{for}\;\; \alpha \rightarrow 0^+, \; \forall\; \lambda>0
\end{equation}
and
\begin{equation*}
0<\underset{\alpha \rightarrow
0^+}{\limsup}\,\frac{s(\lambda)\abs{r_\alpha(\lambda)}}{\rho_2(\alpha)}=\underset{\alpha
\rightarrow
0^+}{\limsup}\,\frac{s(\lambda)\abs{r_\alpha(\lambda)}}{\rho_1(\alpha)}\frac{\rho_1(\alpha)}{\rho_2(\alpha)}.
\end{equation*}
It follows from (\ref{eq:6}) and (\ref{eq:7}) that the $\limsup$
on the right-hand side of the previous expression must be equal to
zero, which is a contradiction. Then,  $\underset{\alpha \to
0^+}{\limsup}\frac{\rho_1(\alpha)}{\rho_2(\alpha)}>0.$ Similarly,
it is shown that $\underset{\alpha \to
0^+}{\limsup}\frac{\rho_2(\alpha)}{\rho_1(\alpha)}>0.$ Since there
exists $\underset{\alpha \to
0^+}{\lim}\frac{\rho_1(\alpha)}{\rho_2(\alpha)}$, we then have
that $0<\underset{\alpha \to
0^+}{\lim}\frac{\rho_1(\alpha)}{\rho_2(\alpha)}<+\infty$ and
$0<\underset{\alpha \to
0^+}{\lim}\frac{\rho_2(\alpha)}{\rho_1(\alpha)}<+\infty$. Then,
$\rho_1\preceq \rho_2$ and $\rho_2\preceq \rho_1$, that is,
$\rho_1 \approx \rho_2$, as we wanted to show.
\end{proof}

\section{Examples}

In this section we present several examples which illustrate the
different qualification levels previously introduced as well as
the relationships between them and with the concept of classical
qualification and the qualification introduced in
\cite{ref:Mathe-Pereverzev-2003}. Although some of these examples
are only of academic interest and nature, they do serve to show
the existence of regularization methods possessing qualification
in each one of the levels introduced in this article.

\medskip
\textbf{Example 1.} Tikhonov-Phillips regularization method
$\{g_\alpha\}$, where $g_\alpha(\lambda)\doteq
\frac{1}{\lambda+\alpha}$ has classical qualification of order
$\mu_0=1$ (\cite{bookEHN}). We will see that $\rho(\alpha)=\alpha$
is optimal qualification in the sense of Definition
\ref{def:calif-3} \textbf{\textit{iii)}}. In fact, for $\lambda
>0$, $r_\alpha(\lambda)=\frac{\alpha}{\alpha+\lambda}$ and if
$\rho(\alpha)=\alpha$ then $s_\rho(\lambda)=\underset{\alpha \to
0^+}{\liminf}
\frac{\rho(\alpha)}{\abs{r_\alpha(\lambda)}}=\underset{\alpha \to
0}{\lim}(\lambda + \alpha)= \lambda >0$, that is, $s_\rho$
verifies (\ref{eq:cond-calif}). Also  since
\begin{equation*}
\frac{s_\rho(\lambda)\abs{r_\alpha(\lambda)}}{\rho(\alpha)}=\frac{\lambda}{\lambda+\alpha}\geq
\frac{1}{2} \quad \forall\; \lambda \in [\alpha,+\infty),
\end{equation*}
we have that $s_\rho$ verifies (\ref{eq:4.50}). From Theorem
\ref{teo:cond-calopt} it then follows that $\rho(\alpha)=\alpha$
is optimal qualification of  $\{g_\alpha\}$.

\medskip
\textbf{Example 2.} Let $\{g_\alpha\}$ be the family of functions
associated to the truncated singular value decomposition (TSVD),
\begin{equation*}
g_{\alpha}(\lambda)\doteq \left\{%
\begin{array}{ll}
    \frac{1}{\lambda}, & \hbox{if $\lambda \in [\alpha, +\infty)$} \\
    0,& \hbox{if $\lambda \in [0,\alpha)$}. \\
\end{array}%
\right.
\end{equation*}
It follows that $\mu_0=+\infty$, where $\mu_0$ is as in Definition
\ref{def:calif-clasica}. Therefore, TSVD does not have classical
qualification. In this case we have that
\begin{equation*}
r_{\alpha}(\lambda)= \left\{%
\begin{array}{ll}
    0, & \hbox{if $\lambda \in [\alpha, +\infty)$} \\
    1,& \hbox{if $\lambda \in [0,\alpha)$}. \\
\end{array}%
\right.
\end{equation*}

Let $h(\alpha)=\alpha$ and $\rho \in \mathcal{O}$. Then
\begin{equation*}
\underset{\lambda \in
[h(\alpha),+\infty)}{\sup}\abs{r_\alpha(\lambda)}=
\underset{\lambda\geq\alpha}{\sup}\abs{r_\alpha(\lambda)}=0\leq
\rho(\alpha) \quad \textrm{for every}\; \alpha \in (0,\alpha_0).
\end{equation*}
Then, it follows from Theorem \ref{teo:cond-caldebil}.a) that any
function $\rho \in \mathcal{O}$ is weak qualification of the
method. However, TSVD does not have strong qualification. In fact,
for any function $\rho \in \mathcal{O}$ we have that
$s_\rho(\lambda)=\underset{\alpha \to 0^+}{\liminf}
\frac{\rho(\alpha)}{\abs{r_\alpha(\lambda)}}=+\infty$ for every
$\lambda >0$. Proposition \ref{teo:cond-calif} implies then that
$\rho$ is not strong qualification of the method. In
\cite{ref:Mathe-Pereverzev-2003} it was observed that TSVD has
arbitrary qualification in the sense of Definition
\ref{def:calif-mathe}.

\medskip
\textbf{Example 3.} For $\alpha \in (0,\alpha_0)$ we define
\begin{equation*}
g_{\alpha}(\lambda)\doteq
\frac{1-e^{-\frac{1}{\alpha}}}{\lambda+e^{-\frac{1}{\alpha}}},\quad
\textrm{for every}\; \lambda \in [0,+\infty).
\end{equation*}
It can be immediately verified that $\{g_\alpha\}$ satisfies the
hypotheses \textit{H1-H3} and therefore is a SRM. Since
$r_{\alpha}(\lambda)= \frac{1+\lambda}{1+\lambda\,
e^{\frac{1}{\alpha}}}$ for all $\lambda \in [0,+\infty)$, it
follows that for every $\mu>0$,
\begin{equation*}
    \frac{\abs{r_\alpha(\lambda)}\lambda^\mu}{\alpha^\mu}=\frac{(1+\lambda)\lambda^\mu}{
\lambda \,
e^{\frac{1}{\alpha}}\alpha^\mu+\alpha^\mu}=o(1)\;\textrm{for} \;
\alpha
    \rightarrow 0^+ \; \textrm{for every}\; \lambda \in [0,+\infty).
\end{equation*}
Then, $\{g_\alpha\}$ does not have classical qualification (more
precisely $\mu_0=+\infty$, where $\mu_0$ is as in Definition
\ref{def:calif-clasica}).

We will now show that $\rho(\alpha)=e^{-\frac{1}{\alpha}}$ is
optimal qualification of  $\{g_\alpha\}$. Since
$s_\rho(\lambda)=\underset{\alpha \to 0^+}{\liminf}
\frac{\rho(\alpha)}{\abs{r_\alpha(\lambda)}}=\frac{\lambda}{1+\lambda}
\in (0,+\infty)$ for every $\lambda >0$, it follows from
Proposition \ref{teo:cond-calif} that $\rho$ is strong
qualification of $\{g_\alpha\}$. Moreover since
\begin{equation*}
\frac{s_\rho(\lambda)\abs{r_\alpha(\lambda)}}{\rho(\alpha)}=\frac{\lambda}{\lambda+e^{-\frac{1}{\alpha}}}\geq
\frac{1}{2} \quad \forall\; \lambda \in
[e^{-\frac{1}{\alpha}},+\infty),
\end{equation*}
it follows that $s_\rho$ verifies (\ref{eq:4.50}). Theorem
\ref{teo:cond-calopt} then implies that
$\rho(\alpha)=e^{-\frac{1}{\alpha}}$ is optimal qualification of
$\{g_\alpha\}$.

\medskip
\textbf{Example 4.} For $\alpha \in (0,\alpha_0)$ with
$\alpha_0<e^{-1}$, define
\begin{equation*}
g_{\alpha}(\lambda)\doteq \frac{1+ (\ln{\alpha})^{-1}}{\lambda
-(\ln\alpha)^{-1}},\quad \textrm{for every}\; \lambda \in
[0,+\infty).
\end{equation*}
Clearly, $\{g_\alpha\}$ satisfies hypotheses \textit{H1-H3} and
therefore is a SRM. Since $r_{\alpha}(\lambda)=
\frac{1+\lambda}{1-\lambda\,\ln \alpha}$ for all $\lambda \in
[0,+\infty)$, it follows that for every $\mu>0$,
\begin{equation*}
    \frac{\abs{r_\alpha(\lambda)}\lambda^\mu}{\alpha^\mu}=\frac{(1+\lambda)\lambda^\mu}{
\alpha^\mu-\lambda \, \alpha^\mu \ln\alpha}\to
+\infty\;\textrm{for} \; \alpha
    \rightarrow 0^+ \; \textrm{for every}\; \lambda \in [0,+\infty).
\end{equation*}
Then, $\mu_0=0$ and therefore $\{g_\alpha\}$ does not have
classical qualification.

However, we will show that $\rho(\alpha)=-(\ln \alpha)^{-1}$ is
optimal qualification of $\{g_\alpha\}$. In fact, since
$s_\rho(\lambda)=\underset{\alpha \to 0^+}{\liminf}
\frac{\rho(\alpha)}{\abs{r_\alpha(\lambda)}}=\underset{\alpha \to
0}{\lim}\frac{\lambda-(\ln
\alpha)^{-1}}{1+\lambda}=\frac{\lambda}{1+\lambda} \in
(0,+\infty)$ for every $\lambda>0$ and
\begin{equation*}
\frac{s_\rho(\lambda)\abs{r_\alpha(\lambda)}}{\rho(\alpha)}=\frac{\lambda}{\lambda-(\ln
\alpha)^{-1}}\geq \frac{1}{2} \quad \forall\; \lambda \in
[-(\ln\alpha)^{-1},+\infty),
\end{equation*}
it follows from Theorem  \ref{teo:cond-calopt} that
$\rho(\alpha)=-(\ln\alpha)^{-1}$ is optimal qualification of
$\{g_\alpha\}$.

\medskip
\textbf{Example 5.} Let $\{g_\alpha\}$ be the Tikhonov-Phillips
regularization method, which, as previously mentioned, it has
classical qualification of order $\mu_0=1$. In Example 1 we saw
that $\rho(\alpha)=\alpha$ is optimal qualification of this method
and therefore it is also weak qualification of it. Since $\alpha
\preceq \alpha^{\frac 1 2}$ it follows from Definition
\ref{def:calif-3}.i) and Observation 2.a) that
$\rho^*(\alpha)=\alpha^{\frac 1 2}$ is also weak qualification.
However, $\rho^*$ is not strong qualification of the method. In
fact, for any $s \in \mathcal{S}$, we have that
\begin{equation*}
    \underset{\alpha \rightarrow
0^+}{\limsup}\,\frac{s(\lambda)\abs{r_\alpha(\lambda)}}{\rho^*(\alpha)}=
\underset{\alpha \rightarrow
0^+}{\limsup}\,\frac{s(\lambda)\;\alpha^{\frac 1 2}}{\alpha +
\lambda}=0\quad \forall\; \lambda>0.
\end{equation*}

\medskip
\textbf{Example 6.} Let $\{g_\alpha\}$ be the SRM defined in
Example 4. This method does not have classical qualification since
$\mu_0=0$. We proved that $-(\ln \alpha)^{-1}$ is optimal
qualification and therefore, it is also weak qualification. Since
$-(\ln \alpha)^{-1}\preceq (-\ln \alpha)^{-\frac 1 2}$, just like
in the previous example, it follows immediately that
$\rho(\alpha)=(-\ln \alpha)^{-\frac 1 2}$ is weak qualification.
Let us show now that $\rho$ is not strong qualification of the
method. For any $s \in \mathcal{S}$, we have that
\begin{equation*}
    \underset{\alpha \rightarrow
0^+}{\limsup}\,\frac{s(\lambda)\abs{r_\alpha(\lambda)}}{\rho(\alpha)}=
\underset{\alpha \rightarrow
0^+}{\limsup}\,\frac{s(\lambda)\;(1+\lambda)}{(1-\lambda \ln
\alpha)\, (- \ln \alpha)^{-\frac 1 2}}=0\quad \forall\; \lambda>0.
\end{equation*}

\bigskip
\bigskip
It is important to observe that if $\rho(\alpha)=\alpha^\mu$ is
strong qualification of a SRM then it follows immediately from the
definition of strong source-order pair that the method has
classical qualification of order $\mu$. The converse, however, is
not true as the next example shows. Hence it is the weak and not
the strong qualification what generalizes the classical notion of
this concept.

\medskip
\textbf{Example 7.} For $\alpha \in (0,\alpha_0)$ with
$\alpha_0<1/2$ define
\begin{equation*}
    h_\alpha(\lambda)\doteq \frac{\alpha}{\alpha+\ln(\frac{\alpha}{\alpha+\lambda})}
\end{equation*}
and
\begin{equation*}
g_{\alpha}(\lambda)\doteq \left\{%
\begin{array}{ll}
    \frac{1-h_\alpha(\lambda)}{\lambda+h_\alpha(\lambda)}, & \hbox{if $\lambda \in [2\alpha, +\infty)$} \\
   \frac{1-h_\alpha(2\alpha)}{2\alpha+h_\alpha(2\alpha)}=\left(2\alpha-\frac{\alpha+2\alpha^2}{\ln{3}}\right)^{-1} ,& \hbox{if $\lambda \in [0,2\alpha)$}. \\
\end{array}%
\right.
\end{equation*}
In this case,
\begin{equation*}
r_{\alpha}(\lambda)\doteq \left\{%
\begin{array}{ll}
    \frac{\alpha(1+\lambda)}{\lambda\ln(\frac{\alpha}{\alpha+\lambda})+\alpha(1+\lambda)}, & \hbox{if $\lambda \in [2\alpha, +\infty)$} \\
   1-\lambda\left(2\alpha-\frac{\alpha+2\alpha^2}{\ln{3}}\right)^{-1} ,& \hbox{if $\lambda \in [0,2\alpha)$}. \\
\end{array}%
\right.
\end{equation*}
One can immediately show that $\{g_\alpha\}$ is a SRM with
classical qualification of order $\mu_0=1$. However,
$\rho(\alpha)=\alpha$ is not strong qualification of the method.
In fact, for any $s \in \mathcal{S}$, we can see that
\begin{equation*}
\frac{s(\lambda)\abs{r_\alpha(\lambda)}}{\alpha}=o(1)\quad
\textrm{for}\;\; \alpha \rightarrow 0^+, \; \forall\; \lambda\geq
0
\end{equation*}
and therefore condition (\ref{eq:no-o}) is not satisfied.

\bigskip
\medskip

SRMs possessing strong but not optimal qualification, have very
peculiar properties. Thus for instance, it is possible to show
that if $\rho$ is strong qualification which is not optimal, then
$\forall \lambda>0$, the function
$\frac{s_\rho(\lambda)|r_\alpha(\lambda)|}{\rho(\alpha)}$ it is
not of bounded variation as a function of $\alpha$ in any
neighborhood of $\alpha=0$. Even so, the following three examples
show the existence of SRM having strong but not optimal
qualification and they show that strong qualification in no case
implies optimal qualification.

\medskip
\textbf{Example 8.} Given $k \in \R^+$, for $\alpha,\lambda>0$
define
\begin{equation*}
    g_\alpha^k(\lambda)\doteq\lambda^{-1}(1-e^{-\frac \lambda \alpha})-\alpha^k\lambda^{-3/2}
    \left|{\textrm{sin}{(\lambda^{\frac 3 2}/\alpha)}}\right|\;,
\end{equation*}
so that $r_\alpha^k(\lambda)=e^{-\frac \lambda \alpha} + \alpha^k
\lambda^{-1/2}\abs{\textrm{sin}{(\lambda^{\frac 3 2}/\alpha)}}.$
It can be immediately checked that $\{g^k_\alpha\}$ is a SRM with
classical qualification of order $k$. With $\rho(\alpha)=\alpha^k$
we have that $\forall\,\lambda>0$,
\begin{eqnarray*}
  s_\rho(\lambda)&=& \underset{\alpha \to 0^+}{\liminf} \frac{\alpha^k}{e^{-\frac{\lambda}{\alpha}}+\alpha^k\lambda^{-1/2}\abs{\textrm{sin}(\lambda^{\frac 3 2}/\alpha)}} \\
  &=& \frac{1}{\underset{\alpha \to 0^+}{\limsup}\left(\alpha^{-k}e^{-\frac \lambda \alpha}+\lambda^{-1/2}\abs{\textrm{sin}(\lambda^{\frac 3 2}/\alpha)}\right)} \\
  &=& \lambda^{1/2}.
\end{eqnarray*}
Since $s_\rho(\lambda)=\lambda^{1/2}\in\mathcal{S}$, from
Proposition \ref{teo:cond-calif} it follows that $(s_\rho,\rho)$
is a strong source-order pair and $\rho(\alpha)=\alpha^k$ is
strong qualification of the method. However, for every
$\lambda>0$,
\begin{eqnarray*}
   \underset{\alpha \to 0^+}{\liminf}\,\frac{s_\rho(\lambda)\abs{r_\alpha(\lambda)}}{\rho(\alpha)}&=&
   \underset{\alpha \to 0^+}{\liminf}\,\left[\lambda^{1/2}\alpha^{-k}e^{-\frac \lambda \alpha}+\abs{\textrm{sin}\left(\frac
   {\lambda^{\frac 3 2}}
   \alpha\right)}\right]\;=\;0.
\end{eqnarray*}
Therefore equation (\ref{eq:4.50}) does not hold and
$\rho(\alpha)=\alpha^k$ is not optimal qualification of the
method.

\medskip
\textbf{Example 9.} For $\alpha,\lambda>0$ define
$g_\alpha(\lambda)$ as follows:
\begin{eqnarray*}
g_\alpha(\lambda)\;\doteq\;\lambda^{-1}(1-e^{-\frac\lambda\alpha})
- e^{-\frac{1}{\sqrt\alpha}}
\lambda^{-3/2}\abs{\sin(\lambda^{\frac 3 2}/\alpha)},
\end{eqnarray*}
so that
\begin{eqnarray*}
r_\alpha(\lambda)\;=\;e^{-\frac\lambda\alpha} +
e^{-\frac{1}{\sqrt\alpha}}\lambda^{-1/2}\abs{\sin(\lambda^{\frac 3
2}/\alpha)}.
\end{eqnarray*}
It can be immediately verified that $\{g_\alpha\}$ is a SRM which
does not have classical qualification ($\mu_0=\infty$). However,
with  $\rho(\alpha)\doteq e^{-\frac{1}{\sqrt\alpha}}$ we have that
\begin{eqnarray*}
s_\rho(\lambda)\;&=&\;\underset{\alpha \to
0^+}{\liminf}\;\frac{\rho(\alpha)}{r_\alpha(\lambda)}\\
&=&\frac{1}{\underset{\alpha \to
0^+}{\limsup}\left[e^{-\frac\lambda\alpha+\frac{1}{\sqrt\alpha}}
+\lambda^{-1/2}|\sin(\lambda^{\frac 3 2}/\alpha)|\;\right]}\\
&=&\lambda^{\frac12}.
\end{eqnarray*}
Since $s_\rho(\lambda)=\lambda^{1/2}\in\mathcal{S}$, by
Proposition \ref{teo:cond-calif} $(s_\rho,\rho)$ is a strong
source-order pair and $\rho(\alpha)=e^{-1/\sqrt\alpha}$ is strong
qualification of the method. However, $\forall\,\lambda>0$ we have
that
\begin{eqnarray*}
\underset{\alpha \to 0^+}{\liminf} \;
\frac{s_\rho(\lambda)|r_\alpha(\lambda)|}{\rho(\alpha)} \;=\;
\underset{\alpha \to
0^+}{\liminf}\left(\lambda^{1/2}e^{\frac{1}{\sqrt\alpha}-\frac\lambda\alpha}
+ |\sin(\lambda^{\frac 3 2}/\alpha)| \right)\;=\;0,
\end{eqnarray*}
and therefore (\ref{eq:4.50}) does not hold and
$\rho(\alpha)=e^{-\frac{1}{\sqrt\alpha}}$ is not optimal
qualification of the method.

\medskip
\textbf{Example 10.} For $0<\alpha <1$ and $\lambda>0$ define
\begin{eqnarray*}
g_\alpha(\lambda)\;\doteq\;\lambda^{-1}(1-e^{-\frac\lambda\alpha})
+ (\textrm{ln}\;\alpha)^{-1}
\lambda^{-3/2}\abs{\sin(\lambda^{\frac 3 2}/\alpha)},
\end{eqnarray*}
so that
\begin{eqnarray*}
r_\alpha(\lambda)\;=\;e^{-\frac\lambda\alpha} -
(\textrm{ln}\;\alpha)^{-1} \lambda^{-1/2}\abs{\sin(\lambda^{\frac
3 2}/\alpha)}.
\end{eqnarray*}
Just like in Examples 8 and 9 it can be easily checked that
$\left\{ g_\alpha\right\}$ is a SRM which does not have classical
qualification ($\mu_0=0$), that
$\rho(\alpha)=\frac{-1}{\textrm{ln}\;\alpha}$ is strong but not
optimal qualification of the method and that $(s_\rho,\rho)$ is a
strong source-order pair with $s_\rho(\lambda)=\lambda^{\frac12}$.

\bigskip
Note that examples 2, 3, 4, 6, 9 and 10 correspond to SRMs which
do not have classical qualification but, however, they do have
generalized qualification, falling in some of its three different
levels. Also  Landweber's method and Showalter's method, which as
previously pointed out do not have classical qualification (in
both cases  $\mu_0=\infty$), are SRMs defined by
$g_\alpha(\lambda) \doteq\frac 1 \lambda (1-(1-\mu\lambda)^{\frac
1
 \alpha})$ (where $\alpha\le 1,\; \mu<\frac{1}{\norm{T}^2})$ and
 $g_\alpha(\lambda)\doteq\frac 1 \lambda
 (1-e^{-\frac{\lambda}{\alpha}})$, respectively. It can be easily
 proved, by using Theorem \ref{teo:cond-caldebil}, that $\rho(\alpha)=(1-\mu\alpha^{\frac12})^\frac 1
 \alpha$ is weak qualification of Landweber's method and
 $\rho(\alpha)=e^{-\frac{1}{\sqrt\alpha}}$ is weak qualification of
 Showalter's method. However, in this last case it
 can be easily shown that
 $\rho(\alpha)=e^{-\frac{1}{\sqrt\alpha}}$ does not satisfy condition
 (\ref{eq:calif-mathe}) and therefore $\rho(\alpha)$
 is not qualification in the sense of Definition \ref{def:calif-mathe}.

\bigskip
The different qualification levels introduced in this article and
the relationships between them are visualized in Figure 1.

\begin{figure}[h]
\begin{center}
 \includegraphics[width=1.00\textwidth]{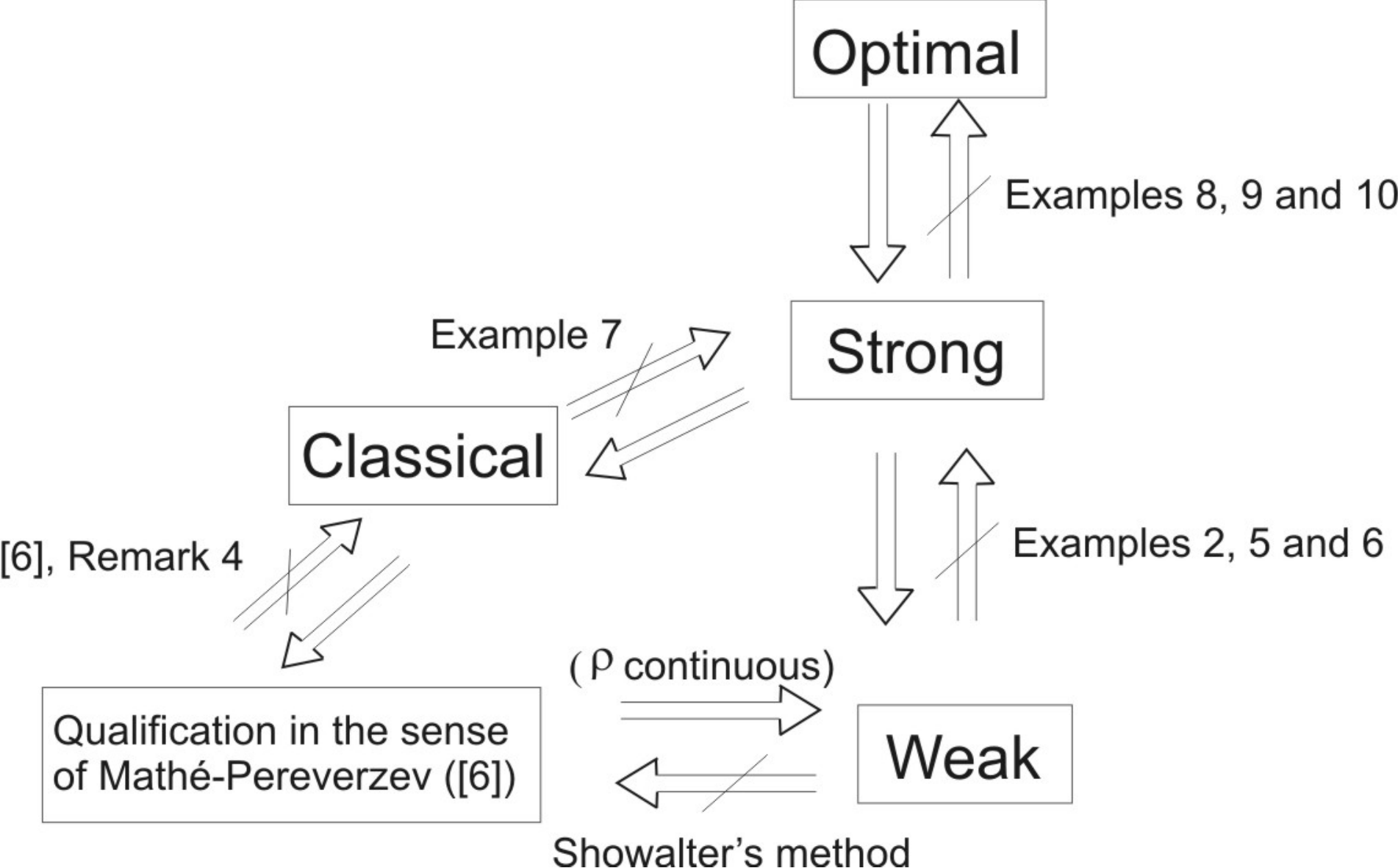}
\vskip .5in
 \includegraphics[width=0.5\textwidth]{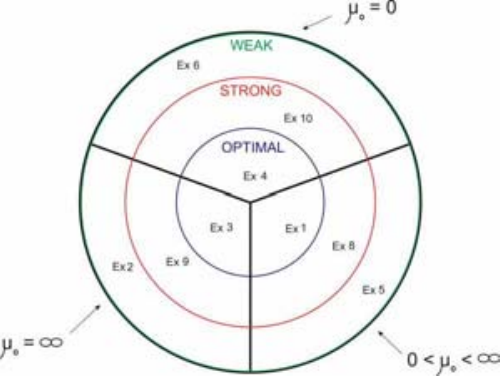}
\end{center}
\caption{Relationships between the different qualification levels,
the classical qualification  and the qualification defined in
\cite{ref:Mathe-Pereverzev-2003}.}
\end{figure}

\section{Orders of convergence, converse results and maximal source sets}

The generalization of the concept of qualification of a SRM
introduced in the previous sections is strongly related with and
it has a broad spectrum of applications in the context of orders
of convergence, converse results and maximal source sets for
inverse ill-posed problems. We present next some results in this
direction. However, we point out that this is not the main
objective of the present article. For that reason, some of this
results will be stated without proof. More detailed results in
this regard will appear in a forthcoming article.

Let $X, Y$ be infinite dimensional Hilbert spaces and
$T:X\rightarrow Y$ a bounded, linear invertible operator such that
$\mathcal{R}(T)$ is not closed. For $s\in\mathcal{S}$, the set
$\mathcal{R}(s(T^\ast T))$, will be referred to as  the ``source
set associated to the function $s$ and the operator $T$". In all
that follows, the hypothesis $s\in \mathcal{S}$ can be replaced by
$s$ continuous on $\sigma(T^\ast T)$ and $s\in \mathcal{M}_0$,
where $\mathcal{M}_0$ is the set of all functions $f:\R
\rightarrow \R^+_0$ which are measurable with respect to the
measures $d\norm{E_{\lambda}x}^2$ for every $x \in X$.

The following direct result, whose proof follows immediately from
the concept of weak source-order pair, states that if the exact
solution $x^\dag$ of the problem $Tx=y$ belongs to the source set
$\mathcal{R}(s(T^\ast T))$ and $(s,\rho)$ is a weak source-order
pair for $\{g_\alpha\}$, then the regularization error
$\norm{R_\alpha y-x^\dag}$ has order of convergence
$\rho(\alpha)$. For brevity reasons we do not give the proof here.

\begin{thm}\label{teo:gen-4.3}
Let $\rho\in\mathcal{O}$ be weak qualification of $\{g_\alpha\}$
and $s \in \mathcal{S}$ such that $(s,\rho)$ is a weak
source-order pair for $\{g_\alpha\}$. If $x^\dag\doteq T^\dag y
\in \mathcal{R}(s(T^\ast T))$ then
$\norm{(R_\alpha-T^\dag)y}=O(\rho(\alpha))$ for $\alpha
\rightarrow 0^+$.
\end{thm}

It is important to note here that the previous result can be
viewed as a generalization of Theorem 4.3 in \cite{bookEHN}, to
the case of SRM with weak qualification and general source sets.
In fact, that result corresponds to the particular case in which
$\{g_\alpha\}$ has classical qualification of order $\mu$.

The following converse result states that if the regularization
error has order of convergence $\rho(\alpha)$ and $(\rho,s)$ is an
order-source pair, then the exact solution belongs to the source
set given by the range of the operator $s(T^\ast T)$.

\begin{thm}\label{teo:gen-411} If $(\rho,s)$ is an order-source
pair for $\{g_\alpha\}$ and
$\norm{(R_\alpha-T^\dag)y}=O(\rho(\alpha))$ for $\alpha\rightarrow
0^+$, then $x^\dag \in \mathcal{R}(s(T^\ast T)).$
\end{thm}

\begin{proof} The proof follows immediately from the definition of
order-source pair for the SRM $\{g_\alpha\}$.
\end{proof}

It is interesting to note that Theorem \ref{teo:gen-411} can also
be viewed as a generalization of Theorem 4.11 in \cite{bookEHN}.
In fact, this corresponds to the particular case in which
$s(\lambda)\doteq\lambda^\mu$ y $\rho(\alpha)\doteq\alpha^\mu$. If
moreover $\rho$ is optimal qualification then the reciprocal of
Theorem \ref{teo:gen-411} also holds. This is proved in the
following theorem.

\begin{thm}
If $\rho$ is optimal qualification of $\{g_\alpha\}$ and $s_\rho
\in \mathcal{S}$, then $\norm{(R_\alpha-T^\dag)y}=O(\rho(\alpha))$
for $\alpha\rightarrow 0^+$ if and only if $x^\dag \in
\mathcal{R}(s_\rho(T^\ast T)).$
\end{thm}
\begin{proof}
Let $\rho$ be optimal qualification of $\{g_\alpha\}$ and $s_\rho
\in \mathcal{S}$. Then by Theorem \ref{teo:unica nu},
$(\rho,s_\rho)$ is an order-source pair for $\{g_\alpha\}$ and
since $\norm{(R_\alpha-T^\dag)y}=O(\rho(\alpha))$ for
$\alpha\rightarrow 0^+$, it follows from Theorem \ref{teo:gen-411}
that $x^\dag \in \mathcal{R}(s_\rho(T^\ast T)).$

Conversely, if $x^\dag \in \mathcal{R}(s_\rho(T^\ast T))$, since
by virtue of Theorem \ref{teo:unica nu} $(s_\rho,\rho)$ is a
strong source-order pair, Theorem \ref{teo:gen-4.3} implies that
$\norm{(R_\alpha-T^\dag)y}=O(\rho(\alpha))$ for $\alpha\rightarrow
0^+$.
\end{proof}

An important result regarding existence and maximality of source
sets is the following: if $\rho$ is strong qualification of a SRM
and $s_\rho \in \mathcal{S}$ it follows from Proposition
\ref{teo:nu<nu-rho} that $\mathcal{R}(s_\rho(T^\ast T))$ is a
maximal source set where $\rho$ is order of convergence of the
regularization error. More precisely we have the following result.

\begin{thm}\label{teo:rango}
Let $\rho \in \mathcal{O}$ be strong qualification of
$\{g_\alpha\}$ such that $s_\rho \in \mathcal{S}$ and $s\in
\mathcal{S}$. If $(s,\rho)$ is a strong source-order pair for
$\{g_\alpha\}$ and $\mathcal{R}(s(T^\ast T))\supset
\mathcal{R}(s_\rho(T^\ast T))$ then $\mathcal{R}(s(T^\ast T))=
\mathcal{R}(s_\rho(T^\ast T))$.
\end{thm}

\begin{proof}
Under the hypotheses of the Proposition \ref{teo:nu<nu-rho}, there
exists $k>0$ such that $s(\lambda)\leq k\,s_\rho(\lambda)$ for
every $\lambda>0$, which implies that $\mathcal{R}(s(T^\ast
T))\subset \mathcal{R}(s_\rho(T^\ast T))$.
\end{proof}

If moreover $\rho$ is optimal qualification the following stronger
result is obtained.

\begin{thm}
If $\rho \in \mathcal{O}$ is optimal qualification of $\{g_\alpha\}$
and $s_\rho \in \mathcal{S}$, then $\mathcal{R}(s_\rho(T^\ast T))$
is the only source set where $\rho$ is order of convergence of the
regularization error of $\{g_\alpha\}$.
\end{thm}

\begin{proof}
This result follows immediately from Theorem \ref{teo:unica nu}.
\end{proof}

\textbf{Examples:}

1. For the Tikhonov-Phillips regularization method the only source
set where $\rho(\alpha)=\alpha$ is optimal qualification is
$\mathcal{R}(s_\rho(T^\ast T))=\mathcal{R}(T^\ast T)$, since in
this case $s_\rho(\lambda)=\lambda$.

2. In Example 3 of Section 3 we saw that
$\rho(\alpha)=e^{-\frac{1}{\alpha}}$ is optimal qualification of
$\{g_\alpha\}$ and $s_\rho(\lambda)=\frac{\lambda}{1+\lambda}$.
Since  $\frac{\lambda}{1+\lambda}\approx \lambda$ it follows that
$\mathcal{R}(s_\rho(T^\ast T))=\mathcal{R}(T^\ast T)$ is the only
source set where $\rho$ is order of convergence of the
regularization error.

3. In Example 8 of the previous section, for $\rho(\alpha)=\alpha^k$
we have that $s_\rho(\lambda)=\lambda^{1/2}$. Since $\rho$  is
strong qualification of this SRM, it follows that
$\mathcal{R}(s_\rho(T^\ast T))=\mathcal{R}(T^\ast T)^{1/2}$ is a
maximal source set where $\rho(\alpha)$ is order of convergence of
the regularization error.

4. As pointed out at the end of Section 3, $\rho(\alpha)=
e^{-\frac{1}{\sqrt{\alpha}}}$ is weak qualification of Showalter's
method. It can be easily shown that for every $s \in \mathcal{S}$,
$(s,\rho)$ is a weak source-order pair for the method. Therefore,
it follows from Theorem \ref{teo:gen-4.3} that the regularization
error $\norm{R_\alpha y-x^\dag}$ has order of convergence
$\rho(\alpha)= e^{-\frac{1}{\sqrt{\alpha}}}$ whenever $x^\dag \in
\underset{s\in \,\mathcal{S}}{\bigcup}\mathcal{R}(s(T^\ast T))$.

5. Same as 4. happens with Landweber's method and
$\rho(\alpha)=(1-\mu\sqrt{\alpha}\,)^\frac 1
 \alpha$.

\section{Conclusions}

In this article we have extended the definition of qualification
for spectral regularization methods introduced by Math\'{e} and
Pereverzev in \cite{ref:Mathe-Pereverzev-2003}. This extension was
constructed bearing in mind the concept of qualification as the
optimal order of convergence of the regularization error that a
method can achieve (\cite{bookEHN}, \cite{Mathe2004},
\cite{ref:Mathe-Pereverzev-2003}, \cite{Vainikko-1982}). Three
different levels of generalized qualification were introduced:
weak, strong and optimal. In particular, the first of these levels
extends the definition introduced in
\cite{ref:Mathe-Pereverzev-2003} and a SRM having weak
qualification which is not qualification in the sense of
Definition \ref{def:calif-mathe} was shown. Sufficient conditions
for a SRM to have weak qualification were provided, as well as
necessary and sufficient conditions for a given order of
convergence to be strong or optimal qualification. Examples of all
three qualification levels were provided and the relationships
between them as well as with the classical concept of
qualification  and the qualification introduced in
\cite{ref:Mathe-Pereverzev-2003} were shown. Several SRMs having
generalized qualification in each one of the three levels and not
having classical qualification were presented. In particular, it
was shown that the well known TSVD, Showalter's and Landweber's
methods do have weak qualification. Finally  several implications
of this theory in the context of orders of convergence, converse
results and maximal source sets for inverse ill-posed problems,
were briefly shown. More detailed results on these implications
will appear in a forthcoming article.

\end{document}